\documentclass{article}
\usepackage[final,nonatbib]{neurips_2018}

\usepackage{wrapfig}
\let\labelindent\relax
\usepackage{enumitem,soul}
\usepackage{rotating}
\usepackage{chngcntr}
\usepackage{apptools}
\usepackage{listings}
\usepackage{graphicx}
\usepackage{graphics}
\usepackage{amssymb}
\usepackage{amsmath}
\usepackage{color,cite}
\usepackage{soul}
\usepackage{xspace}
\usepackage{algpseudocode}
\usepackage{algorithm,algorithmicx}
\usepackage{bbm}
\usepackage{comment}
\usepackage{subfig}
\usepackage{psfrag}
\usepackage{tikz}
\usetikzlibrary{math}
\usetikzlibrary{shapes,arrows}
\usetikzlibrary{arrows.meta}
\tikzset{>={Latex[width=2.5mm,length=2.5mm]}}
\usetikzlibrary{positioning}
\tikzstyle{block}=[draw opacity=0.7,line width=1.4cm]
\usepackage{rotating}
\usepackage{graphicx}
\usepackage{sidecap}
\usepackage{diagbox}

\title{\LARGE \bf
Multi-Agent Maximization of a Monotone Submodular Function via Maximum Consensus}

\author{%
  Navid Rezazadeh\\
  Mechanical and Aerospace Eng. Dept.\\
  University of California Irvine\\
  Irvine, CA \\
  \texttt{nrezazad@uci.edu}
  \And
  Solmaz S. Kia\\
  Mechanical and Aerospace Eng. Dept.\\
  University of California Irvine\\
  Irvine, CA \\
  \texttt{solmaz@uci.edu} \\
}

\newcommand{\real}{{\mathbb{R}}}

\newcommand{\realnonnegative}{{\mathbb{R}}_{\ge 0}}

\renewcommand{\baselinestretch}{1.1}

\newcommand{\vect}[1]{\boldsymbol{\mathbf{#1}}}

\newcommand{\SUM}[2]{\sum_{#1}^{#2}}
 \newcommand{\boxend}{\hfill \ensuremath{\Box}}

\newtheorem{thm}{Theorem}[section]
\newtheorem{prop}{Proposition}[section]
\newtheorem{rem}{Remark}[section]

\newtheorem{lem}{Lemma}[section]

\newenvironment{proof}[1][Proof]{\begin{trivlist}
\item[\hskip \labelsep {\bfseries #1}]}{\boxend\end{trivlist}}

\makeatletter
\renewcommand*{\@opargbegintheorem}[3]{\trivlist
      \item[\hskip \labelsep{\bfseries #1\ #2}] \textbf{(#3)}\ \itshape}
\makeatother

\newcommand{\oprocendsymbol}{\hbox{$\bullet$}}
\newcommand{\oprocend}{\relax\ifmmode\else\unskip\hfill\fi\oprocendsymbol}



\begin{document}
\maketitle
\thispagestyle{plain}
\pagestyle{plain}

\begin{abstract}                     
Constrained submodular set function maximization problems often appear in multi-agent decision-making problems with a discrete feasible set. A prominent example is the problem of multi-agent mobile sensor placement over a discrete domain.
However, submodular set function optimization problems are known to be NP-hard. In this paper, we consider a class of submodular optimization problems that consists of maximization of a monotone and submodular set function subject to a uniform matroid constraint over a group of networked agents that communicate over a connected undirected graph. Our objective is to obtain a distributed suboptimal polynomial-time algorithm that enables each agent to obtain its respective policy via local interactions with its neighboring agents. 
Our solution is a fully distributed gradient-based algorithm using the multilinear extension of the submodular set functions and exploiting a maximum consensus scheme. This algorithm results in a policy set that when the team objective function is evaluated at worst case the objective function value is in $1-1/{\textup{e}}-O(1/T)$ of the optimal solution. An example demonstrates our results.
\end{abstract}

\section{Introduction}
In recent years, optimal multi-agent sensor placement/dispatch to cover areas for sampling/observation objectives has been of great interest in the robotic community to solve various coverage, exploration, and monitoring problems. Depending on the problem setting, the planning space can be continuous~\cite{SM-FB:06,NZ-CGC-XY:17,YK-CGC:17,NZ-XYSBA-CGC:18} or discrete~\cite{NR-SSK-Conf:19,XD-MG-RP-FB:20,SW-CGC:20,SA-EF-LSL:14}.  
In this paper, our interest is in planning in discrete space.  Many optimal discrete policy-making for multi-agent sensor placement problems appear in the form of maximizing a utility for the team~\cite{NM-RH:18,AH-MG-HV-UT:20,VT-AJ-GJP:16,MS-SB-HV:10,STJ-SLS:15,NR-SSK-Conf:19}. A particular subset of these problems that we are interested in is
in the form of a set-valued function maximization problem described by
\begin{align}\label{eq::mainPrbIntro}
    \underset{\mathcal{R} \in \mathcal{I}\subset\mathcal{P}}{\textup{max}}f(\mathcal{R}),
\end{align}
where $\mathcal{I} \subset 2^{\mathcal{P}}$ is the admissible set of subsets originated from the ground policy set $\mathcal{P}$, which is the union of the local policy set of the agents; see Fig.~\ref{fig:city_partition} for an example.

Problem~\eqref{eq::mainPrbIntro} is known to be NP hard~\cite{GLN-LAW-MLF:78}. However, when the objective function is monotone increasing and submodular set function, it is well-known that the sequential greedy algorithm~\cite{MLF-GLN-LAW:78} delivers a polynomial-time suboptimal solution with guaranteed optimality bound. For large scale submodular maximization problems, reducing the size of the problem through approximations~\cite{KW-RI-JB:14} or using several processing units leads to a faster sequential greedy algorithm, however, with some sacrifices on the optimality bound~\cite{BM-AK-RS-AK:13,BM-MZ-AK:16,RK-BM-SV-AV:15,PSR-OS-MF:20}. 

When the set that the agents can choose their policy from is the same for all agents and the agents are homogeneous, the feasible set $\mathcal{I}$ of~\eqref{eq::mainPrbIntro}  can be spanned by a uniform matroid constraint\footnote{see Section~\ref{sec::submodul} for a formal definition of uniform matroid constraint.}.
For problems with uniform matroid constraint, the sequential greedy algorithm delivers a suboptimal solution whose utility value is no worse than $1-\frac{1}{\textup{e}}$ times the optimal utility value~\cite{GLN-LAW-MLF:78}. On the other hand, when the local policy set of the agents is disjoint and/or the agents are heterogeneous, and each agent has to choose one policy from its local set, the feasible set $\mathcal{I}$ of~\eqref{eq::mainPrbIntro} can be spanned by the so-called partitioned matroid constraint\footnote{see Section~\ref{sec::submodul} for a formal definition of partitioned matroid constraint.}. For problems with partitioned matroid constraint, the sequential greedy algorithm delivers a suboptimal solution whose utility value is no worse than $\frac{1}{2}$ times the optimal utility value~\cite{MLF-GLN-LAW:78}. For problems with partitioned matroid constraint, more recently, a suboptimal solution with a better optimality gap is proposed in~\cite{JV:08,AAB-BM-JB-AK:17,AM-HH-AK:20,OS-MF:20} using the multilinear continuous relaxation of a submodular set function. The multilinear continuous relaxation of a submodular set function defined on the vertices of the $n$-dimensional hypercube facilitates achieves better optimality bounds for the maximization problem~\eqref{eq::mainPrbIntro} by allowing to use a continuous gradient-based optimization algorithm. That is, the relaxation transforms the main discrete problem into a continuous optimization problem with linear constraints. The optimality bound achieved by the algorithm proposed in~\cite{JV:08} is  $1-\frac{1}{\textup{e}}$ bound for a partitioned matroid constraint, which outperforms the sequential greedy algorithm whose optimally bound is $1/2$. However, this solution requires a central authority to find the optimal solution. In sensor placement problems, when the agents are self-organizing autonomous mobile agents with communication and computation capabilities, it is desired to solve the optimization problem~\ref{eq::mainPrbIntro} in a distributed way, without involving a central authority. A distributed multilinear extension based continuous algorithm that uses an average consensus scheme between the agents to solve~\eqref{eq::mainPrbIntro} subject to partitioned matroid constraint is proposed in~\cite{AR-AA-BS-JGP-HH:19}. However, the proposed algorithm assumes that access to the exact multilinear extension of the utility function is available, whose calculation is exponentially hard. 

In this paper, motivated by the improved optimality gap of multilinear continuous relaxation based algorithms, we develop a distributed implementation of the algorithm of~\cite{JV:08} over a connected undirected graph to obtain a suboptimal solution for~\eqref{eq::mainPrbIntro} subject to a partitioned matroid constraint. In this algorithm, to manage the computational cost of constructing the multilinear extension of the utility function, we use a sampled based evaluation of the multilinear extension and propose a gradient-based algorithm, which uses a maximum consensus message passing scheme over the communication graph. Our algorithm uses a fully distributed rounding technique to compute the final solution of each agent. Through rigorous analysis, we show that our proposed distributed algorithm achieves, with a known probability, a $1-\frac{1}{\textup{e}}-O(1/T)$ optimality bound, where $T$ is the number of times agents communicated over the network. A numerical example demonstrates our results.

The organization of this paper is as follows. In Section~\ref{sec::Notation} we introduce the notation used throughout the paper. Section~\ref{sec::submodul} provides the necessary background on submodularity and submodular functions. We introduce our proposed algorithm in Section~\ref{sec::Main}. A demonstration of our results is provided in  Section~\ref{sec::numerical}. Appendix A gives the proof of our main results in Section~\ref{sec::Main}. Appendix B gives the auxiliary results that we use in establishing the proof of our main results.

\section{Notation}\label{sec::Notation}
We denote the vectors with bold small font.  The $p^{\textup{th}}$ element of vector $\vect{x}$ is denoted by $x_p$. We denote the inner product of two vectors $\vect{x}$ and $\vect{y}$ with appropriate dimensions by $\vect{x}.\vect{y}$. Furthermore, we use $\vect{1}$ as a vector of ones with appropriate dimension. We denote sets with capital curly font. 
For a $\mathcal{P}=\{1,\cdots,n\}$ and a vector $\vect{x} \in \realnonnegative^n$ with $0 \leq x_p \leq 1$, the set $\mathcal{R}_{\vect{x}}$ is a random set where $p \in\mathcal{P}$ is in  $\mathcal{R}_{\vect{x}}$ with the probability $x_p$. Hence, we call such vector $\vect{x}$ as membership probability vector. Furthermore, for $\mathcal{S} \subset \mathcal{P}$, $\vect{1}_{\mathcal{S}} \in \{0,1\}^{n}$ is the vector whose $p^{\textup{th}}$ element is $1$ if $p \in \mathcal{S}$ and $0$ otherwise; we call $\vect{1}_{\mathcal{S}}$ the membership indicator vector of set $\mathcal{S}$. $|x|$ is the absolute value of  $x \in \real$. By overloading the notation, we also use $|\mathcal{S}|$ as the cordiality of set $\mathcal{S}$. We denote a graph by $\mathcal{G}(\mathcal{A},\mathcal{E})$ where $\mathcal{A}$ is the node set and $\mathcal{E} \subset \mathcal{A} \times \mathcal{A}$ is the edge set. 
$\mathcal{G}$ is undirected if and only $(i,j) \in \mathcal{E}$ means that agents $i$ and $j$ can exchange information. An undirected graph is connected if there is a path from each node to every other node in the graph. 
We denote the set of the neighboring nodes of node $i$ by $\mathcal{N}_i=\{j\in\mathcal{A}|(i,j) \in \mathcal{E} \}$. We also use $d(\mathcal{G})$ to show the diameter of the graph.

Given a set $\mathcal{F}\subset \mathcal{X}\times \real$ and an element $(p,\alpha)\in \mathcal{X}\times \real$ we define the addition operator $\oplus$ as $\mathcal{F}{\oplus}\{(p,\alpha)\}=\{(u,\gamma)\in \mathcal{X}\times \real\,|\,(u,\gamma) \in \mathcal{F}, u\not = p \} \cup \{(u,\gamma+\alpha)\in \mathcal{X}\times \real\,|\, (u,\gamma) \in \mathcal{F}, u = p \} \cup \{(p,\alpha)\in \mathcal{X}\times \real\,|\,(p,\gamma) \not \in \mathcal{F} \}$. Given a collection of sets $\mathcal{F}_j\in \mathcal{X}\times \real$, $j\in\mathcal{N}$, we define the max-operation over these collection as $\underset{j\in \mathcal{N}}{\textup{MAX}} \,\,\mathcal{F}_j= \{(u,\gamma)\in X\times \real| (u,\gamma) \in \bar{\mathcal{F}} \text{~s.t.~} \gamma =  \underset{(u,\alpha) \in \bar{\mathcal{F}}}{\textup{max}} \alpha\}$, where $\bar{\mathcal{F}}=\bigcup_{j\in \mathcal{N}} \mathcal{F}_j.$

\section{Submodular Functions}\label{sec::submodul}

A set function $f:2^{\mathcal{P}} \to \realnonnegative$ is monotone increasing if for sets $\mathcal{P}_1 \subset \mathcal{P}_2 \subset \mathcal{P}$,
\begin{align*}
    f(\mathcal{P}_1)\leq f(\mathcal{P}_2),
\end{align*}
holds. Furthermore, defining $\Delta_f(p|\mathcal{S}) = f(\mathcal{S} \cup \{p\})-f(\mathcal{S}),\,\, \mathcal{S} \subset \mathcal{P}$, the set function $f$ is submoular if for $p \in \mathcal{P} \setminus \mathcal{P}_2$,
\begin{align}\label{eq::submodularity}
    \Delta_f(p|\mathcal{P}_1)\geq \Delta_f(p|\mathcal{P}_2),
\end{align}
holds, which shows the diminishing return of a set function. Furthermore, a set function is normalized if $f(\emptyset)=0$. 

In what follows without loss of generality we assume that the ground set $\mathcal{P}$ is equal to $\{1,\cdots,n\}$. Submodular functions are functions assigning values to all subsets of a finite set $\mathcal{P}$. Using the set membership indicator, equivalently,
we can regard them as functions on the boolean hypercube, $f:\{1,0\}^{n}\to\real$. For a submodular function $f:2^{\mathcal{P}} \to \realnonnegative$, its \emph{multilinear extension}  $F:[0,1]^{n} \to \realnonnegative$ in the continuous space is 
\begin{align}\label{eq::F_determin}
    F(\vect{x}) = \SUM{\mathcal{R} \subset \mathcal{P}}{} f(\mathcal{R}) \prod_{p \in \mathcal{R}}^{} x_p \prod_{p \not\in \mathcal{R}}^{}(1-x_p),\quad~~ \vect{x} \in [0,1]^{n}.
\end{align}
The multilinear extension $F$ of set function $f$ is a unique multilinear function agreeing with f on the vertices of the hypercube. The multilinear extension function has the following properties.
\begin{lem}[See~\cite{JV:08}]\label{lem::prop_F} If $f$ is non-decreasing, then $\frac{\partial F}{\partial x_p}\geq0$ for all $p\in\mathcal{P}$, everywhere in $[0,1]^n$ (monotonicity of $F$). If $f$ is submodular, then $\frac{\partial^2 F}{\partial x_p x_q}\leq 0$ for all $p,q\in\mathcal{P}$, everywhere in $[0,1]^n$.
\end{lem}
 An alternative  way to perceive $F$ is to observe that it is indeed equivalent to \begin{align}\label{eq::F_stoc}
   F(\vect{x})=\mathbb{E}[f(\mathcal{R}_{\vect{x}})],  
 \end{align} where $\mathcal{R}_{\vect{x}}\subset \mathcal{P}$ is a set where each element $p\in\mathcal{R}_{\vect{x}}$ appears independently with the probabilities $x_p$. Then, it can be shown that taking the derivatives of $F(\vect{x})$ yields 
\begin{align}\label{eq::firstDer}
\frac{\partial F}{\partial x_p} (\vect{x})= \mathbb{E}[f(\mathcal{R}_{\vect{x}} \cup \{p\})-f(\mathcal{R}_{\vect{x}} \setminus \{p\})],
\end{align}
and
\begin{align}\label{eq::secondDer}
& \frac{\partial^2 F}{\partial x_p \partial x_q}(\vect{x}) = \mathbb{E}[f(\mathcal{R}_{\vect{x}} \cup \{p,q\})-f(\mathcal{R}_{\vect{x}} \cup \{q\} \setminus \{p\})\nonumber\\
&\quad\quad \quad \quad-f(\mathcal{R}_{\vect{x}} \cup \{p\} \setminus \{q\})+f(\mathcal{R}_{\vect{x}} \setminus \{p,q\})].
\end{align}

Constructing $F(\vect{x})$ in~\eqref{eq::F_determin} requires the knowledge of $f(\mathcal{R})$ for all $\mathcal{R} \subset \mathcal{P}$, which becomes computationally intractable when the size of ground set $\mathcal{P}$ increases. However, with the stochastic interpretation~\eqref{eq::F_stoc} of the multilinear extension and its derivatives, if enough samples are drawn according to $\vect{x}$, we can obtain an estimate of $F(\vect{x})$ with a reasonable computational cost.  We can use Chernoff-Hoeffding's inequality to quantify the quality of this estimation given the number of samples.
\begin{thm}[Chernoff-Hoeffding's inequality~\cite{WH:94}]\label{thm::sampling}
Consider a set of $K$ independent random variables $X_1,...,X_K$ where $a<X_i<b$. Letting $S_K=\SUM{i=1}{K} X_i$, then
\begin{align*}
    \mathbb{P}[|S_K - \mathbb{E}[S_K]|>\delta] < 2  \textup{e}^{-\frac{2\delta ^2}{K(b-a)^2}}.
\end{align*}
\end{thm}
Problems such as sensor placement with a limited number of sensors and ample places of possible sensor placement can be formulated as a set value optimization  \eqref{eq::mainPrbIntro} where the feasible set constraint $\mathcal{I}$ is in the form of \emph{matroid}. A matroid is defined as the pair $\mathcal{M}=\{\mathcal{P},\mathcal{I}\}$ with $\mathcal{I} \subset 2^{\mathcal{P}}$ such that
\begin{itemize}
    \item $\mathcal{B} \in \mathcal{I}$ and $\mathcal{A} \subset \mathcal{B}$ then $\mathcal{A} \in \mathcal{I}$
    \item $\mathcal{A},\mathcal{B} \in \mathcal{I}$ and $|\mathcal{B}|>|\mathcal{A}|$, then there exists $x \in \mathcal{B} \setminus \mathcal{A}$ such that $\mathcal{A}\cup{x} \in \mathcal{I}$
\end{itemize}
One of the known matroids is uniform matroid $\mathcal{I} = \{ \mathcal{R} \subset \mathcal{P}, |\mathcal{R}|\leq k \}.$
A sensor placement scenario where there are $k$ sensors to be placed in locations selected out of possible locations $\mathcal{P}$ can be described by a uniform matroid. On the other hand, when there is a set of agents $\mathcal{A}$, each with multiple sensors $k_i$, $i\in\mathcal{A}$, that they can place in a set of specific and non-overlapping part $\mathcal{P}_i$ of a discrete space, this constrained can be described by the partitioned matroid $\mathcal{M}=\{\mathcal{P},\mathcal{I}\}$, where $\mathcal{I} = \{ \mathcal{R} \subset \mathcal{P}, |\mathcal{R} \cap \mathcal{P}_i|\leq k_i\}$ with  $\bigcup\limits_{i=1}^{N}  \mathcal{P}_i =  \mathcal{P}$ and $ \mathcal{P}_i \cap  \mathcal{P}_j = \emptyset,\,\, i\neq j$. 

A \emph{matroid polytop} is a convex hull defined as 
\begin{align*}
    P(\mathcal{M})=\{ \vect{x} \in \realnonnegative^{|\mathcal{P}|},\forall \mathcal{Q} \in \mathcal{P}, \SUM{p \in \mathcal{Q}}{} x_p \leq r_{\mathcal{M}}(\mathcal{Q})\}
\end{align*}
with $r_{\mathcal{M}}$ to be the rank function of the matroid defined as
\begin{align*}
    r_{\mathcal{M}}(\mathcal{Q}) = \textup{max}\{|\mathcal{S}|,\mathcal{S}\subset \mathcal{Q},\mathcal{S} \in \mathcal{I}\}.
\end{align*}
For the partition matroid with $k_i=1$ which is the focus of this paper, the convex hull is such that for $\vect{x}=[\vect{x}_1^\top,...,\vect{x}_N^\top]^\top$ with $\vect{x}_i \in \realnonnegative^{|\mathcal{P}_i|}$ associated with membership probability of policies in $\mathcal{P}_i$, we have $\vect{1}^\top \vect{x}_i  \leq 1$, i.e., 
\begin{align}\label{eq::convexhull}P(\mathcal{M})=\{ \vect{x} \in \realnonnegative^{|\mathcal{P}|}\,\big|\,\sum\nolimits_{p=1}^{|\mathcal{P}_i|}x_{ip}\leq1, \forall i\in\mathcal{A}\}.
\end{align}

\begin{figure}
\centering
    {\includegraphics[width=.90\textwidth]{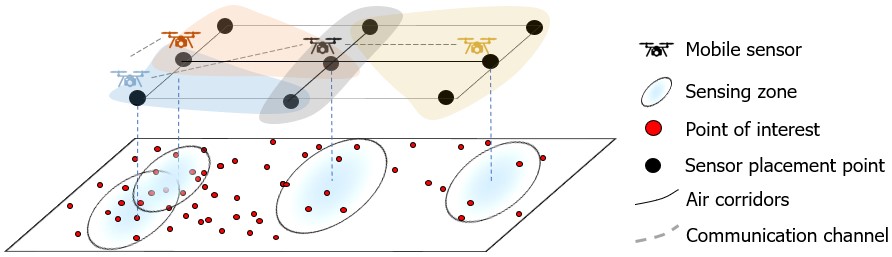}}%
    \caption{{\small   
    Let the policy set of each mobile sensor $i\in\mathcal{A}$ be $\mathcal{P}_i=\{(i,p)|p\in \mathcal{B}_i \}$, where $\mathcal{B}_i \subset \mathcal{B}$ is the set of the allowable sensor placement points for agent $i\in\mathcal{A}$ out of all the sensor placement points $\mathcal{B}$. Note that by definition, for any two agent $i,j\in\mathcal{A}$, $\mathcal{P}_i\cap \mathcal{P}_j=\emptyset$. The sensors are heterogeneous, in the sense that the size of their sensing zone is different. The objective is to place the sensors in points in $\mathcal{B}$ such that the total number of the observed points of interest is maximized. The utility function, the sum of observed points, is known to be a monotone and increasing submodular function of the agent's sensing zone~\cite{CC-AK:04}. This sensor placement problem can be formalized as the optimization problem~\eqref{eq::mainProblem}. The agents are communicating over a connected undirected graph and their objective is to obtain their respective placement points by interacting only with their communicating neighbors.}}
    \label{fig:city_partition}%
\end{figure}

\section{Problem Statement}
We consider a group of $\mathcal{A}=\{1,...,N\}$ with communication and computation capabilities interacting over a connected undirected graph $\mathcal{G}(\mathcal{A},\mathcal{E})$. These agents aim to solve in a distributed manner the optimization problem
\begin{subequations}\label{eq::mainProblem}
\begin{align}
    &\underset{\mathcal{R} \in \mathcal{I}}{\textup{max}}f(\mathcal{R}) \quad \textup{s.t.}\\
    \mathcal{I} = \big\{ &\mathcal{R} \subset \mathcal{P}\,\big|\,\, |\mathcal{R} \cap \mathcal{P}_i|\leq 1,~ i\in\mathcal{A}\big\},\label{eq::mainProblem_b}
\end{align}
\end{subequations}
where utility function $f:2^\mathcal{P}\to\real_{\geq0}$ is  monotone increasing and submodular set function over the discrete policy space  $\mathcal{P} =\bigcup\limits_{i=1}^{N}  \mathcal{P}_i $, with $\mathcal{P}_i$ being the policy space of agent $i\in\mathcal{A}$, which is only known to agent $i$. In
this paper, we work in the value oracle model where the
only access to the utility function is through a black box
returning $f(\mathcal{R})$ for a given set $\mathcal{R}$. Every agent can obtain the value of the utility function at any subset $\mathcal{R}\in\mathcal{P}$.
The constraint~\eqref{eq::mainProblem_b} is the partitioned matroid $\mathcal{M}=\{\mathcal{P},\mathcal{I}\}$, which ensures that only one policy per agent is selected from each local policy set $\mathcal{P}_i$, $i\in\mathcal{A}$. An example application scenario of our problem of interest is shown in Fig.~\ref{fig:city_partition}. Without loss of generality, to simplify notation, we assume that the policy space is $\mathcal{P} = \{1,...,n\},$ and is sorted agent-wise with $1 \in \mathcal{P}_1 $ and $n \in \mathcal{P}_N$.
 
 Finding the optimal solution $\mathcal{R}^\star \in \mathcal{I}$ of~\eqref{eq::mainProblem} even in central form is NP-Hard \cite{MC-GC:84}. The computational time of finding the optimizer set increases exponentially with $N$~\cite{NR-SSK:19}. The well-known sequential greedy algorithm finds a suboptimal solution $\mathcal{R}_{\textup{SG}}$ for~\eqref{eq::mainProblem} with the optimality bound of $f(\mathcal{R}_{\textup{SG}})\geq \frac{1}{2} f(\mathcal{R}^\star)$, i.e., a $1/2$-approximation at worst case~\cite{MC-GC:84}. More recently, by using the multilinear extension of the submodular utility functions, Vondr{\'a}k~\cite{JV:08} developed a randomized centralized continuous greedy
algorithm which achieves a $(1 -1/\text{e})-O(1/T)$-approximate for set value optimization~\eqref{eq::mainProblem} in the value oracle model, see Algorithm~\ref{alg:the_alg_cent}\footnote{Pipage rounding is a polynomial time algorithm which moves a fractional point $\vect{x}$ on a hypercube to a integral point $\hat{\vect{x}}$ on the same hypercube such that $f(\hat{\vect{x}})\geq f(\vect{x})$}. Our objective in this paper is to develop a distributed implementation of Algorithm~\ref{alg:the_alg_cent} to solve~\eqref{eq::mainProblem} for when agents interact over a connected graph $\mathcal{G}$. Recall that in our problem setting, every agent $i\in\mathcal{A}$ can evaluate the utility function for a given $\mathcal{R}\subset\mathcal{P}$ but it has access only to its own policy space $\mathcal{P}_i$.

\begin{algorithm}[t]
\caption{Practical implementation of the continuous greedy process~\cite{JV:08}.} 

\label{alg:the_alg_cent}
{
\begin{algorithmic}[1]
\State $\vect{x} \gets \vect{0},\,\, \mathbf{\text{Init:}} ~ t \gets 1$,
\While{$t\leq T$}
\State Draw $K$ samples of $\mathcal{R}$ from $\mathcal{P}$ according to $\vect{x}$
\For{$p\in \mathcal{P}$}
 \State Estimate $w_p \sim \mathbb{E}[f(\mathcal{R}_{\vect{x}} \cup \{p\})-f(\mathcal{R}_{\vect{x}} \setminus \{p\})]$
 \EndFor
    \State Solve for $\mathcal{R}^\star=\underset{\mathcal{R} \in \mathcal{I}}{\textup{argmax}} \,\, \SUM{p\in \mathcal{R}}{} w_p.$
\State Update membership vector as $\vect{x} \gets \vect{x}+\frac{1}{T}\vect{1}_{\mathcal{R}^\star}$
\State $t \gets t+1$
\EndWhile
\State Use Pipage rounding to convert the fractional solution $\vect{x}$ to an integral solution.
\end{algorithmic}
{Note: $\vect{1}_{\mathcal{R}^\star}$ is the $\vect{v}$ in~\eqref{eq::xv}. }
}
\end{algorithm}  
 
\section{A polynomial-time distributed multi-agent randomized continuous greedy algorithm}\label{sec::Main}
Our proposed distributed multi-agent randomized continuous greedy algorithm to solve the set value optimization problem~\eqref{eq::mainProblem} over a connected graph $\mathcal{G}$ is given in Algorithm~\ref{alg:the_alg_decen_2}, whose convergence guarantee and suboptimality gap is given in Theorem~\ref{thm::main} below. To provide the insight into construction of Algorithm~\ref{alg:the_alg_decen_2}, 
we first review the idea behind the central suboptimal solution via Algorithm~\ref{alg:the_alg_cent} following~\cite{JV:08}. We also provide some intermediate results that will be useful in establishing our results.

\subsection{A short overview of the central continuous greedy process}
As we mentioned earlier the continuous multilinear extension is a relaxation strategy that extends a submodular function $f(\mathcal{R})$, which is defined on the vertices of the $n-$dimensional hypercube $\{0,1\}^n$ to a continuous multilinear function $F$ defined on $[0, 1]^n$. The two functions evaluate identically for any vector $\vect{x}\in [0, 1]^n$ that is the membership indicator vector of a set $\mathcal{R}\in\mathcal{P}$. Then, by way of a process that runs continuously, depending only
on local properties of $F$, we can produce a point $\vect{x} \in P(\mathcal{M})$\ to approximate
the optimum $OPT = \underset{ \vect{x} \in P(\mathcal{M})}{\textup{max}}F(\vect{x})$ (here recall~\eqref{eq::convexhull}). The proposal in~\cite{JV:08} is to move in the direction of a vector constrained by
$P(\mathcal{M})$ which maximizes the local gain. To understand the logic behind Algorithm~\ref{alg:the_alg_cent} let us review the conceptual steps of this continuous greedy process.~\cite{JV:08} views the process as a particle starting at $\vect{x}(0)=\vect{0}$ and
following the flow 
\begin{align}\label{eq::VNabF}
\frac{\text{d}\vect{x}}{\text{d}t}=\vect{v}(\vect{x})~~\text{where}~~ \vect{v}(\vect{x})=\underset{\vect{v}\in P(\mathcal{M})}{\arg\max}(\vect{v}.\nabla F(\vect{x})).
\end{align} 
over a unit time interval $[0,1]$.
We note that $\vect{x}(t)$ for $t\in[0,1]$ is contained in $P(\mathcal{M})$, since it is a convex combination of vectors in  $P(\mathcal{M})$.
Now consider a point $\vect{x}$ along the trajectory of our flow~\eqref{eq::VNabF} and assume that $\vect{x}^\star$ is the true optimum $OPT=F(\vect{x}^\star)$. Now consider a direction $\vect{v}^\star=\max\{\vect{x}^\star-\vect{x},\vect{0}\}$, which is a nonnegative vector. Because  $0\leq\vect{v}^\star\leq \vect{x}^\star\in P(\mathcal{M})$, then $\vect{v}^\star\in P(\mathcal{M})$. By virtue of Lemma~\ref{lem::prop_F}, $F$ is monotone increasing, therefore, using $\max\{\vect{x}^\star-\vect{x},\vect{0}\}=\max\{\vect{x}^\star,\vect{x}\}-\vect{x}$ we have $F(\vect{x}+\vect{v}^\star)=F(\max\{\vect{x}^\star,\vect{x}\})\geq F(\vect{x}^\star)$. However, $\vect{x}+\vect{v}^\star$ does not belong to $P(\mathcal{M})$ necessarily. 
Therefore, lets consider $F(\vect{x}+\zeta\vect{v}^\star)$ for $\zeta\geq0$. It can be shown that $F(\vect{x}+\zeta\vect{v}^\star)$ is concave in $\zeta$ and $\frac{\text{d}F}{\text{d}\zeta}$ is non-increasing. Thus, it can be established that $F(\vect{x}+\vect{v}^\star)-F(\vect{x})\leq \frac{\text{d}F}{\text{d}\zeta}|_{\zeta=0}=\vect{v}^\star.\nabla F(\vect{x})$. But, since $\vect{v}^\star\in P(\mathcal{M})$, and $ \vect{v}\in P(\mathcal{M})$ that is used to generate $\vect{v}$ maximizes $\vect{v}.\nabla F(\vect{x})$, we can write $\vect{v}.\nabla F(\vect{x})\geq\vect{v}^\star.\nabla F(\vect{x})\geq  F(\vect{x}+\vect{v}^\star)-F(\vect{x})\leq OPT-F(\vect{x})$. Now we note that $\frac{\text{d}F}{\text{d}t}=\vect{v}(\vect{x}).\nabla F(\vect{x})\geq OPT-F(\vect{x}(t)).$ Therefore, given $\vect{x}(0)=\vect{0}$, using the Comparison Lemma~\cite{HKK-JWG:02}, we can conclude the $F(\vect{x})\geq (1-\text{e}^{-t})OPT$, and thus $\vect{x}(1)\in P(\mathcal{M})$ and also $F(\vect{x}(1))\geq (1-1/\text{e})OPT$.  
In the second stage of the algorithm, the fractional solution $\vect{x}(1)$ is rounded to a point in $\{0,1\}^n$ by use of Pipage rounding method, see~\cite{AAA-MIS:04} for more details about Pipage rounding. The aforementioned exposition is the conceptual design behind the continuous greedy process. The algorithm \ref{alg:the_alg_cent}is a practical implementation achieved by the use of a numerical iterative process  
\begin{align}\label{eq::xv}
    \vect{x}(t+1)=\vect{x}(t)+\frac{1}{T}\vect{v}(t),
\end{align}
and use of sampling to compute $\nabla F(\vect{x})$ and consequently $\vect{v}(t)$, see~\cite{JV:08} for more details. In what follows, we explain a practical distributed implementation of the the continuous greedy process, which is realized as Algorithm~\ref{alg:the_alg_decen_2}, and is inspired by this central solution.

\begin{algorithm}[t]
\caption{Discrete Distributed implementation of the continuous greedy process.} 

\label{alg:the_alg_decen_2}
{\small
\begin{algorithmic}[1]
\State $\mathbf{\text{Init:}} ~\bar{\mathcal{P}} \gets \emptyset,\,\, \mathcal{F}_i \gets \emptyset,\,\, t \gets 1$,
\While{$t\leq T$}
\For{$i\in \mathcal{A}$}
\State Draw $K_i$ sample policy sets  $\mathcal{R}$ such that $q \in \mathcal{R}$ with the probability $\alpha$ for all $(q,\alpha)\in\mathcal{F}_i$.
\For{$p\in \mathcal{P}_i$}
 \State Compute $w^i_p \sim \mathbb{E}[f(\mathcal{R} \cup \{p\})-f(\mathcal{R}\setminus \{p\})]$ using the policy sample sets of step $4$.
\EndFor
\State Solve for $p^\star=\underset{p \in \mathcal{P}_i}{\textup{argmax}} \,\, w^i_p.$
\State 
$\mathcal{F}_i^-\gets \mathcal{F}_i \oplus \{(p^\star,\frac{1}{T})\}.$
\State Broadcast $\mathcal{F}_i^{-}$ to the neighbors $\mathcal{N}_i$.
\State 
$\mathcal{F}_i\gets\underset{j \in \mathcal{N}_i \cup \{i\}}{\textup{MAX}}\mathcal{F}_j^-$
\EndFor
\State $t \gets t+1$.
\EndWhile
\For{$i\in \mathcal{A}$}
    \State Sample one policy $\bar{p} \in \mathcal{P}_i$ using $\mathcal{F}_i$
    \State  $\bar{\mathcal{P}} \gets \bar{\mathcal{P}} \cup \{\bar{p}\}$ 
\EndFor
\end{algorithmic}
}
\end{algorithm}

\subsection{Design and analysis of the distributed continuous greedy process}
We start off our design and analysis of the distributed continuous greedy process by introducing our notation and the set of computations that agents carry out locally using their local information and interaction with their neighbors. The algorithm is performed over the finite time horizon of $T$ steps. Let $\mathcal{F}_i(t)\subset \mathcal{P}
\times [0,1] 
$ be the information set of agent $i$ at time step $t$, initialized at $\mathcal{F}_i(0)=\emptyset$. For, any couple $(p,\alpha)\in\mathcal{F}_i(t)$ the first element is the policy and the second element is the corresponding membership probability. We let $\vect{x}_i(t)\in\real^n$ (recall $|\mathcal{P}|=n$) be the local copy of the membership probability of our suboptimal solution of~\eqref{eq::mainProblem} at agent $i$ at time step $t$, defined according to 
\begin{align}\label{eq::x_iconstruction}
    x_{ip}(t)=\begin{cases}\alpha, & (p,\alpha) \in \mathcal{F}_i(t),\\0&\text{otherwise}, \end{cases}\qquad p\in\mathcal{P}.
\end{align}
Recall that $\mathcal{P} = \{1,...,n\}$ and it is sorted agent-wise with $1 \in \mathcal{P}_1 $ and $n \in \mathcal{P}_N $. Hence, $\vect{x}_i(t)=[\hat{\vect{x}}_{i1}^\top(t),\cdots,\vect{x}_{ii}^\top(t),\cdots,\hat{\vect{x}}_{iN}^\top(t)]^\top$ where $\vect{x}_{ii}(t)\in \realnonnegative^{|\mathcal{P}_i|}$ is the membership probability vector of agent $i$'s own policy at iteration $t$, while $\hat{\vect{x}}_{ij}(t) \in \realnonnegative^{|\mathcal{P}_j|}$ is the local  estimate of the membership probability vector of agent $j$ by agent $i$. 
At time step $t$ agent $i$ solves the optimization problem
\begin{align}\label{eq::updateVect}
\widetilde{\vect{v}}_i(t) = \underset{\vect{y} \in P_i(\mathcal{M})}{\textup{argmax}} \vect{y}.\widetilde{\nabla F}(\vect{x}_i(t)) 
\end{align}
 with
\begin{align}\label{eq::updateVectSpace}P_i(\mathcal{M})=\Big\{ [\vect{y}_1^\top,\cdots,\vect{y}_N^\top]^\top  \in &\realnonnegative^{n}\,\Big|\,\vect{1}^\top.\vect{y}_i\leq 1,\nonumber\\& ~~~\vect{y}_{j}=\vect{0},~j\in\mathcal{A}\backslash\{i\} \Big\}.
\end{align}
The term $\widetilde{\nabla F}(\vect{x}_i(t))$ is the estimate of $\nabla F(\vect{x}_i(t))$ which is calculated by taking $K_i$ samples of set $\mathcal{R}_{\vect{x}_i(t)}$ according to membership probability vector $\vect{x}_i(t)$. Recall~\eqref{eq::firstDer}. Hence, $\frac{\partial F}{\partial x_p}$ is estimated by averaging $f(\mathcal{R}_{\vect{x}_i(t)} \cup \{p\})-f(\mathcal{R}_{\vect{x}_i(t)} \setminus \{p\})$  over the samples. We denote the $p$th element of $\widetilde{\nabla F}(\vect{x}_i(t))$ by 
$w_p^i$, and represent it by  
\begin{equation}\label{eq::bar_w}
    w_p^i\sim \mathbb{E}[f(\mathcal{R}_{\vect{x}_i(t)} \cup \{p\})-f(\mathcal{R}_{\vect{x}_i(t)} \setminus \{p\})].
\end{equation}
We note that given the definition of $P_i(\mathcal{M})$, to compute~\eqref{eq::updateVect}, we only need $ w_p^i$ for $p\in\mathcal{P}_i$. 
\begin{rem}[{Local computation of $w_p^i$,  $p\in\mathcal{P}_i$, by agent $i$}]{
Given the definition~\eqref{eq::x_iconstruction}, we note that $w_p^i$, an estimate of $\frac{\partial F}{\partial x_p}$ can be obtained from drawing $K_i$ sample policy sets  $\mathcal{R}$ such that $q \in \mathcal{R}$ with the probability $\alpha$ for all $(q,\alpha)\in\mathcal{F}_i$ and using \begin{equation}\label{eq::decoupled_wp}
    w^i_p \sim \mathbb{E}[f(\mathcal{R} \cup \{p\})-f(\mathcal{R}\setminus \{p\})],\quad  p\in\mathcal{P}_i.
\end{equation} }
\end{rem}

Let each agent \emph{propagate} its local variable according to 
\begin{align}\label{eq::updateRule1}
    \vect{x}_i^-(t+1) = \vect{x}_i(t)+\frac{1}{T}\widetilde{\vect{v}}_i(t).
\end{align}
One can make a conceptual connection between~\eqref{eq::updateRule1} and the practical implementation of~\eqref{eq::VNabF} discussed earlier. Because the propagation is only based on local information of agent $i$, next, each agent, by interacting with its neighbors, updates its propagated $\vect{x}_i^-(t+1)$ by element-wise maximum seeking
\begin{align}\label{eq::updateRule2}
    \vect{x}_i(t+1) = \underset{j \in \mathcal{N}_i \cup \{i\}}{\textup{max}} \vect{x}_j^-(t+1).
\end{align}
Lemma~\ref{lem::maxVect} below shows that, as one expects, $\vect{x}_{ii}(t+1)=\vect{x}_{ii}^-(t+1)$, i.e., the corrected component of $\vect{x}_{i}$ corresponding to agent $i$ itself is the propagated value maintained at agent $i$, and not the estimated value of any of its neighbors. 

\begin{lem}\label{lem::maxVect}{\it
Assume that the agents follow the distributed Algorithm~\ref{alg:the_alg_decen_2}. Let $\bar{\vect{x}}(t) = \underset{i \in \mathcal{A}}{\textup{max}} \,\, \vect{x}_i(t) $ where $\vect{x}_i$ is given in~\eqref{eq::x_iconstruction}. Moreover, 
Then, $\bar{\vect{x}}(t)=[\vect{x}_{11}^\top(t),\cdots,\vect{x}_{NN}^\top(t)]^\top$ at any time step $t$.}
\end{lem}
The proof of Lemma~\ref{lem::maxVect} is given in Appendix A. We note that it follows from Lemma~\ref{lem::maxVect} that 
\begin{align}
\vect{x}_{jj}(t)=\vect{x}_{jj}^{-}(t), \quad j\in\mathcal{A}.
\end{align} 

In the distributed Algorithm~\ref{alg:the_alg_decen_2} at each time step $t$ each agent has its own belief on the probabilities of the policies that is not necessarily the same as the belief of the other agents. The following result establishes the difference between the belief of the agents.
\begin{prop} \label{thm::convergance}{\it
Let agents follow the distributed Algorithm \ref{alg:the_alg_decen_2}. Then,  the vectorized membership probability $\vect{x}_i(t)$ for each agent $i \in \mathcal{A}$ realized from $\mathcal{F}_i(t)$ satisfy
\begin{subequations}\label{eq::Pro51}
\begin{align}
  & 0\leq \frac{1}{N}\vect{1}.(\bar{\vect{x}}(t)-\vect{x}_i(t)) \leq \frac{1}{T} d(\mathcal{G}),\label{eq::Pro51_a}\\
    &\bar{\vect{x}}(t+1) - \bar{\vect{x}}(t) = \frac{1}{T} \SUM{i \in \mathcal{A}}{} \widetilde{\vect{v}}_i(t),\label{eq::Pro51_b}\\
    &\frac{1}{N}\vect{1}.(\bar{\vect{x}}(t+1) - \bar{\vect{x}}(t)) =\frac{1}{T}.\label{eq::Pro51_c}
\end{align}
\end{subequations}
}
\end{prop}
The proof of Proposition~\ref{thm::convergance} is given in the Appendix A.
Because $f$ is normal and monotone increasing, we have the guarantees that $w_p^i\geq0$. Therefore, without loss of generality, we know that one realization of $\widetilde{\vect{v}}_i(t)$ in~\eqref{eq::updateVect} corresponds to $\vect{1}_{\{p^\star\}}$ where
\begin{align}\label{eq::p_star}
    p^\star=\underset{p \in \mathcal{P}_i}{\arg\max} \,\, w_p^i.
\end{align}
Next, let each agent $i\in\mathcal{A}$ propagate its information set according to
\begin{align}\label{eq::F_prop}
    \mathcal{F}_i^-(t+1)=\mathcal{F}_i(t) \oplus \left\{(p^\star,\frac{1}{T})\right\},
\end{align}
and update it using a local interaction with its neighbors according to 
\begin{align}\label{eq::F_update}
    \mathcal{F}_i(t+1)=\underset{j \in \mathcal{N}_i \cup \{i\}}{\textup{MAX}}\mathcal{F}_j^-(t+1).
\end{align}
By definition to $\oplus$ and $MAX$ operators, we have the guarantees that if $(p,\alpha_1)$, then there exists no $\alpha_2\neq\alpha_1$ that $(p,\alpha_2)\in\mathcal{F}_i$.
\begin{lem}\label{lem::cont_disc}{\it
For $\widetilde{\vect{v}}_i(t)=\vect{1}_{\{p^\star\}}$, $\vect{x}^-_i(t+1)$ and $\vect{x}_i(t+1)$ computed from, respectively,~\eqref{eq::updateRule1} and~\eqref{eq::updateRule2} are the same as $\vect{x}^-_i(t+1)$ and $\vect{x}_i(t+1)$ constructed from, respectively, $\mathcal{F}_i^-(t+1)$ and $\mathcal{F}_i(t+1)$  using~\eqref{eq::x_iconstruction}.}
\end{lem}
\begin{proof}[Proof of Lemma \ref{lem::cont_disc}]
The proof follows trivially from the definition of the operator $\oplus$  and~\eqref{eq::x_iconstruction}.
\end{proof}
Initialized by $\mathcal{F}_i(0)=\emptyset$, $i\in\mathcal{A}$,~\eqref{eq::p_star},~\eqref{eq::F_prop}, and~\eqref{eq::F_prop} where $w_p^i$ is computed via~\eqref{eq::decoupled_wp} constitute a distributed iterative process, formally stated by Algorithm~\ref{alg:the_alg_decen_2}, that runs for $T$ steps. At the end of these $T$ steps, as stated in Algorithm~\ref{alg:the_alg_decen_2}, each agent $i\in\mathcal{A}$, obtains its suboptimal policy $\bar{\mathcal{P}}_i$ by sampling one policy $\bar{p} \in \mathcal{P}_i$ with the probability given by $\vect{x}_{ii}(T)$, where for $p\in\mathcal{P}_i$,  $$x_{ip}(T)=\begin{cases}\alpha, &  (p,\alpha) \in \mathcal{F}_i(T), \\0&\text{otherwise}.\end{cases}$$

 The following result, whose proof is available in the Appendix A, gives the convergence guarantee and suboptimality gap of Algorithm~\ref{alg:the_alg_decen_2}.

\begin{thm}[Convergence guarantee and suboptimality gap of Algorithm~\ref{alg:the_alg_decen_2}]\label{thm::main}
Let $f:2^{\mathcal{P}} \to \realnonnegative$ be normalized, monotone increasing and submodular set function. Let $\mathcal{S}^\star$ to be the optimizer of problem~\eqref{eq::mainProblem}. Then, the admissible policy set $\bar{\mathcal{P}}$, the output of distributed Algorithm \ref{alg:the_alg_decen_2}, satisfies
\begin{align*}
   \left (1\!-\!\frac{1}{\textup{e}}\right)\left(1\!-\!\left(2 N^2d(\mathcal{G})\!+\!\frac{1}{2}N^2+N\right)\frac{1}{T}\right)f(\mathcal{S}^\star)\!\leq\! \mathbb{E}[f(\bar{\mathcal{P}})],
\end{align*}
 with the probability of $\left(\prod_{i\in\mathcal{A}} (1-2\textup{e}^{-\frac{1}{8T^2}K_j})^{|\mathcal{P}_i|}\right)^T$.
\end{thm}
By simplifying the probability statement and dropping the higher order terms, the optimality gap grantee of Theorem~\ref{thm::main} holds with the probability of at least $1-2T\,n\, \textup{e}^{-\frac{1}{8T^2}\underline{K}}, \,\, \underline{K}=\min\{K_1,\cdots,K_N\}$; note that $1-2T\,n\, \textup{e}^{-\frac{1}{8T^2}\underline{K}}\leq \left(\prod_{i\in\mathcal{A}} (1-2\textup{e}^{-\frac{1}{8T^2}K_j})^{|\mathcal{P}_i|}\right)^T$. Then, it is clear that the probability improves as $T$ and the $\underline{K}$, the number of the samples collected by agents, increase.  
\begin{rem}[Extra communication for improved optimality gap]
{\it
Replacing the update step~\eqref{eq::updateRule2} with $\vect{x}_i(t+1)=\vect{y}_i(d(\mathcal{G}))$ where  $\vect{y}_i(0)=\vect{x}_i^-(t+1)$ and 
\begin{align*}
    \vect{y}_i(m) = \underset{j \in \mathcal{N}_i \cup \{i\}}{\textup{max}} \vect{y}_j(m-1), \quad m\in\{1,\cdots, d(\mathcal{G})\},
\end{align*}
i.e., starting with $\vect{x}_i^-(t+1)$ and recursively repeating the update step~\eqref{eq::updateRule2} using the output of the previous recursion for $d(\mathcal{G})$ times, 
each agent $i\in\mathcal{A}$ arrives at $\vect{x}_i(t+1)=\bar{\vect{x}}(t+1)$ (recall Lemma~\ref{lem::maxVect}). 
Hence, for this revised implementation, following the proof of Theorem~\ref{thm::main}, we observe that~\eqref{eq::mainthm3} is replaced by 
$\left|\frac{\partial F}{\partial x_p}(\bar{\vect{x}}(t)) - \frac{\partial F}{\partial x_p}({\vect{x}}_i(t))\right|=0$, which consequently, leads to 
\begin{align}\label{eq::imp_opt_bound}
   \left (1\!-\!\frac{1}{\textup{e}}\right)\left(1\!-\!\left(\!\frac{1}{2}N^2+N\right)\frac{1}{T}\right)f(\mathcal{S}^\star)\!\leq\! \mathbb{E}[f(\bar{\mathcal{P}})],
\end{align}
with the probability of $\left(\prod_{i\in\mathcal{A}} (1-2\textup{e}^{-\frac{1}{8T^2}K_j})^{|\mathcal{P}_i|}\right)^T$. This improved optimality gap is achieved by $(d(\mathcal{G})-1)T$ extra communication per agent. The optimality bound~\eqref{eq::imp_opt_bound} is the same bound that is achieved by the centralized algorithm of~\cite{JV:08}. To implement this revision, Algorithm~\ref{alg:the_alg_decen_2}'s step 11 (equivalent to~\eqref{eq::F_prop}) should be replaced by $ \mathcal{F}_i=\mathcal{H}_i(d(\mathcal{G}))$, where $\mathcal{H}_i(0)= \mathcal{F}_i^{-}$, and 
\begin{align}\label{eq::F_update}
    \mathcal{H}_i(m)=\underset{j \in \mathcal{N}_i \cup \{i\}}{\textup{MAX}}\mathcal{H}_j^-(m-1), \quad m\in\{1,\cdots,d(\mathcal{G})\}.
\end{align}
}
\end{rem}

\begin{figure}[t]
\centering
\subfloat[Case 1]{
    {\includegraphics[width=.20\textwidth]{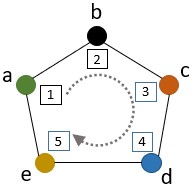}}}~
\subfloat[Case 2]{
    {\includegraphics[width=.20\textwidth]{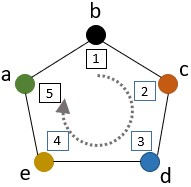}}}~
\subfloat[Case 3]{
    {\includegraphics[width=.20\textwidth]{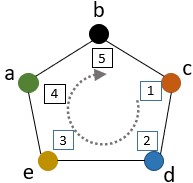}}}\\ 
\subfloat[Case 4]{
    {\includegraphics[width=.20\textwidth]{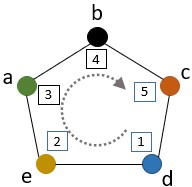}}}~
\subfloat[Case 5]{
    {\includegraphics[width=.20\textwidth]{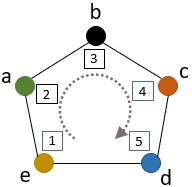}}}~
    \subfloat[Case 6]{
    {\includegraphics[width=.20\textwidth]{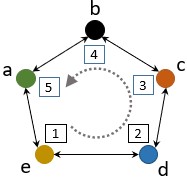}}}
    \\ 
\subfloat[]{
    {\includegraphics[width=.35\textwidth]{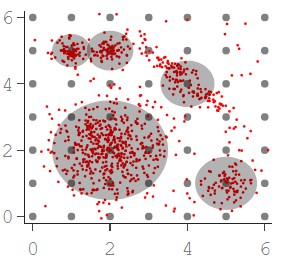}}}~
\subfloat[]{
    {\includegraphics[width=.35\textwidth]{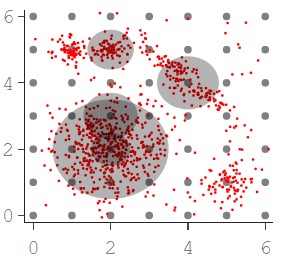}}}
    \caption{{\small Plots (a)-(f) show 6 different $\mathtt{SEQ}$s used in the sequential greedy algorithm. Plot (g) shows the outcome of using Algorithm~\ref{alg:the_alg_decen_2} whereas plot (h) shows the outcome of the sequential greedy algorithm when $\mathtt{SEQ}$ in Case 1 (plot (a)) is used. }}%
    \label{fig:ex1}%
\end{figure}
\section{Numerical Example}\label{sec::numerical}
Consider the multi-agent sensor placement problem introduced in Fig.~\ref{fig:city_partition} for $5$ agents for which $\mathcal{B}_i=\mathcal{B}$, i.e., each agent is able to move to any of the sensor placement points. This sensor placement problem is cast by optimization problem~\eqref{eq::mainProblem}. The field is a 6 unit by 6 unit square and the feasible sensor locations are the 6 by 6 grid in the center square of the field, see Fig.~\ref{fig:ex1}. The points of interest are spread around the map (small red dots in Fig.~\ref{fig:ex1}) in the total number of $900$. The sensing zone of the agents $\mathcal{A}=\{a,b,c,d,e\}$ are circles with radii of respectively $\{0.5,0.6,0.7,0.8,1.5\}$. The agents communicate over a ring graph as shown in Fig.~\ref{fig:ex1}. We first solve this problem using our proposed distributed Algorithm~\ref{alg:the_alg_decen_2}. The results of the simulation for different iteration and sampling numbers are shown in Table~\ref{table:2}. The algorithm produces good results at a modest number of iteration and sampling numbers (e.g. see $T=20$ and $K=500$). Fig.~\ref{fig:ex1}(g) shows the result of the deployment using Algorithm~\ref{alg:the_alg_decen_2}. 
Next, we solve this algorithm using the sequential greedy algorithm~\cite{NR-SSK:19} in a decentralized way by first choosing a route $\mathtt{SEQ}={\scriptsize \fbox{1}\to\fbox{2}\to\fbox{3}\to\fbox{4}\to\fbox{5}}$ that visits all the agents, and then giving $\mathtt{SEQ}$ to the agents so they follow $\mathtt{SEQ}$ to share their information in a sequential manner. Figure~\ref{fig:ex1}(a)-(f) gives 6 possible $\mathtt{SEQ}$ denoted by the semi-circular arrow inside the networks. The results of running the sequential greedy algorithm over the sequences in Fig.~\ref{fig:ex1}(a)-(f) is shown in Table~\ref{table:1}. 
\begin{table}
\begin{center}
\small
\begin{tabular}{ |c|c|c|c|c|c|c|c| } 
\hline 
\diagbox[]{ $T$ }{K} & 10000 & 500 & 100 & 50 & 10 & 5 & 1\\
\hline 
100 & 768 & 768 & 718 & 710 & 718 & 716 & 696 \\\hline 
20 & 768 & 768 & 718 & 710 & 726 & 716 & 696 \\\hline 
10 & 661 & 640 & 657 & 640 & 634 & 602 & 551 \\ \hline 
5 & 630  & 630 & 634 & 626 & 583 & 608 & 540 \\ \hline 
1 & 456  & 456 & 456 & 456 & 456 & 456 & 456\\ \hline 
\end{tabular}
\end{center}
\caption{\small The outcome of Algorithm~\ref{alg:the_alg_decen_2} for different iteration and sampling numbers. }
\label{table:2}
\end{table}
\begin{table}[h]
\small
\begin{center}
\begin{tabular}{ |c||c|c|c|c|c|c| } 
\hline 
 Case & 1 & 2 & 3 & 4 & 5 & 6\\ 
 \hline
 Utility & 634 & 704 & 699 & 640 & 767 & 760\\
 \hline 
\end{tabular}
\end{center}
\caption{\small Outcome of sequential greedy algorithm.}
\label{table:1}
\end{table}
What stands out about the sequential greedy algorithm is that the choice of sequence can affect the outcome of the algorithm significantly. 
 We can attribute this inconsistency to the heterogeneity of the sensors' measurement zone. We can see that when sensor $e$ is given the last priority to make its choice, the sequential greedy algorithm acts poorly. This can be explained by agents with smaller sensing zone picking high-density areas but not being able to cover it fully, see Fig.~\ref{fig:ex2}(h) which depicts the outcome of a sequential greedy algorithm using the sequence in Case 1. A simpler example justifying this observation is shown in Fig.~\ref{fig:ex2} with the two disjoint clusters of points and two sensors. One may suggest to sequence the agents from high to low sensing zone order, however this is not necessarily the best choice as we can see in Table~\ref{table:1}; the utility of case 6 is less than case 5 (the conjecture of sequencing the agents from strongest to weakest is not valid). Moreover, this ordering may lead to a very long $\mathtt{SEQ}$ over the communication graph. Interestingly, this inconsistency does not appear in solutions of Algorithm~\ref{alg:the_alg_decen_2} where the agents intrinsically are overcoming the importance of a sequence by deciding the position of the sensors over a time horizon of communication and exchanging their information~set.

\begin{figure}[t]
\centering
{
    {\includegraphics[width=.75\textwidth]{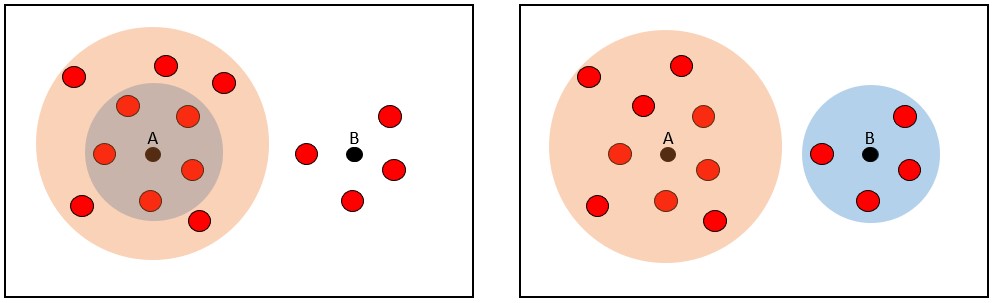}}}~
    \caption{{\small or The sequential greedy algorithm, when the blue agent chooses first assigns both the blue and the orange agents to point A resulting in inferior performance compared to the case that the orange agent chooses first. In the later case, orange agent gets A and the blue agent gets B, which is indeed the optimal solution.}}%
    \label{fig:ex2}%
\end{figure}

\section{Conclusion}
We proposed a distributed suboptimal algorithm to solve the problem of maximizing an monotone increasing submodular set function subject to a partitioned matroid constraint. Our problem of interest was motivated by optimal multi-agent sensor placement problems in discrete space. Our algorithm was a practical decentralization of a multilinear extension based algorithm that achieves $(1-1/\textup{e}-O(1/T))$ optimally gap, which is an improvement over $1/2$ optimality gap that the well-known sequential greedy algorithm achieves. In our numerical example, we compared the outcome obtained by our proposed algorithm with that of a decentralized sequential greedy algorithm which is constructed from assigning a priority sequence to the agents. We showed that the outcome of the sequential greedy algorithm is inconsistent and depends on the sequence. However, our algorithm's outcome due to its iterative nature intrinsically tended to be consistent, which in some ways also explains its better optimally gap over the sequential greedy algorithm. Our future work is to study the robustness of our proposed algorithm to message dropout. 
\subsubsection*{Acknowledgments}
This work is supported by NSF award IIS-SAS-1724331.
\bibliographystyle{ieeetr}%

\section*{Appendix A: Proof of the Results in Section~\ref{sec::Main}}
\begin{proof}[Proof of Lemma \ref{lem::maxVect}]
Since $f$ is a monotone increasing and submodular set function, we have $f(\mathcal{R}_{\vect{x}_i(t)} \cup \{p\})-f(\mathcal{R}_{\vect{x}_i(t)} \setminus \{p\})\geq0$ and hence $\widetilde{\nabla F}(\vect{x}_i(t))$ has positive entries  $\forall i \in \mathcal{A}$. This results in the optimization \eqref{eq::updateVect} subject to vector space \eqref{eq::updateVectSpace} to output vector $\widetilde{\vect{v}}_i(t) \in P_i(\mathcal{M})$ that has entries greater or equal to zero. Hence, according to the propagation and update rule \eqref{eq::updateRule1} and \eqref{eq::updateRule2}, we can conclude that $\vect{x}_{ii}(t)$ has increasing elements and only agent $i$ can update it and other agents only copy this value as $\hat{\vect{x}}_{ji}(t)$. Therefor, we can conclude that $\hat{x}_{jip}(t)\leq x_{iip}(t)$ for all $p \in \mathcal{P}_i$ which concludes the proof. 
\end{proof}

\begin{proof}[Proof of Proposition \ref{thm::convergance}]
$f$ is a monotone increasing and submodular set function therefor $f(\mathcal{R}_{\vect{x}_i(t)} \cup \{p\})-f(\mathcal{R}_{\vect{x}_i(t)} \setminus \{p\})\geq0$ and hence $\widetilde{\nabla} F(\vect{x}_i(t))$ has positive entries  $\forall i \in \mathcal{A}$. Then, because $\widetilde{\vect{v}}_j(t) \in P_j(\mathcal{M})$, it follows from~\eqref{eq::updateVect} that $\widetilde{\vect{v}}_j(t)$ has non-negative entries, $\widetilde{{v}}_{jp}(t)\geq0$, which satisfy $\SUM{ p \in \mathcal{P}_j }{}\widetilde{{v}}_{jp}(t)=1$. Therefore, it follows from~\eqref{eq::updateRule1} and Lemma~\ref{lem::maxVect} that 
\begin{align}\label{eq::oneStepMax}
    \vect{1}.\vect{x}_{jj}(t+1)=\vect{1}.\vect{x}_{jj}(t)+\frac{1}{T},\quad j\in\mathcal{A}.
\end{align}
Using~\eqref{eq::oneStepMax}, we can also write
\begin{align}\label{eq::backwardAdd}
    \vect{1}.\vect{x}_{jj}(t)=\vect{1}.\vect{x}_{jj}(t-d(\mathcal{G}))+\frac{1}{T}d(\mathcal{G}),\quad j\in\mathcal{A}.
\end{align}
Furthermore, it follows from Lemma~\eqref{lem::maxVect} that 
for all $\forall p \in \mathcal{P}_j$ and any $i\in\mathcal{A}\backslash\{j\}$ we can write
\begin{align}\label{eq::elementInEq}
   x_{jjp}(t) \geq\hat{x}_{ijp}(t).
\end{align}
Also, since every agent $i\in\mathcal{A}\backslash\{j\}$ can be reached from agent $j\in\mathcal{A}$ at most in $d(\mathcal{G})$ hops, it follows from the propagation and update laws~\eqref{eq::updateRule1} and~\eqref{eq::updateRule2}, for all $\forall p \in \mathcal{P}_j$, for any $i\in\mathcal{A}\backslash\{j\}$ that
\begin{align}\label{eq::elementInEq2}
\hat{x}_{ijp}(t) \geq x_{jjp}(t-d(\mathcal{G})).
\end{align}
Thus, for any $j\in\mathcal{A}$ and $i\in\mathcal{A}\backslash\{j\}$,~\eqref{eq::elementInEq} and~\eqref{eq::elementInEq2} result in
\begin{align}\label{eq::backwardIn}
    \vect{1}.\vect{x}_{jj}(t) \geq \vect{1}.\hat{\vect{x}}_{ij}(t) \geq \vect{1}.\vect{x}_{jj}(t-d(\mathcal{G})).
\end{align}
Next, we can use~\eqref{eq::backwardAdd} and~\eqref{eq::backwardIn} to write
\begin{align}\label{eq::backwardInSep}
    \vect{1}.\vect{x}_{jj}(t) \geq \vect{1}.\hat{\vect{x}}_{ij}(t) \geq \vect{1}.\vect{x}_{jj}(t) - \frac{1}{T}d(\mathcal{G}),
\end{align}
for $j\in\mathcal{A}$ and $i\in\mathcal{A}\backslash\{j\}$. Using~\eqref{eq::backwardInSep} for any $i\in\mathcal{A}$ we can write 
\begin{align}\label{eq::backwardDouInEq}
    \SUM{j \in \mathcal{A}}{} \vect{1}.\vect{x}_{jj}(t) \geq
   \vect{x}_{ii}(t)+ \SUM{j \in \mathcal{A} \backslash \{i\}}{}  \vect{1}.\hat{\vect{x}}_{ij}(t) \geq
    \SUM{j \in \mathcal{A}}{} \vect{1}.\vect{x}_{jj}(t) - \frac{1}{T}Nd(\mathcal{G}).
\end{align}
Then, using Lemma~\ref{lem::maxVect}, from \eqref{eq::backwardDouInEq} we can write
\begin{align*}
    \vect{1}.\bar{\vect{x}}(t) \geq \vect{1} .\vect{x}_i(t) \geq  \vect{1}.\bar{\vect{x}}(t) - \frac{1}{T}Nd(\mathcal{G}),
\end{align*}
which ascertains~\eqref{eq::Pro51_a}. Next, note that from Lemma~\ref{lem::maxVect}, we have $\vect{x}_{jj}(t)=\vect{x}_{jj}^{-}(t)$ for any $j\in\mathcal{A}$. Then, using~\eqref{eq::updateRule1} and invoking Lemma~\ref{lem::maxVect}, we  obtain~\eqref{eq::Pro51_b},
which, given~\eqref{eq::oneStepMax}, also ascertains~\eqref{eq::Pro51_c}.
\end{proof}

\begin{proof}[Proof of Theorem \ref{thm::main}] Knowing that $\left|\frac{\partial^2F}{\partial x_p \partial x_q}\right| \leq f(\mathcal{S}^\star)$ from Lemma \ref{lm::secondDer} and~\eqref{eq::Pro51_c}, 
it follows from Lemma~\ref{lm::aux3} that
\begin{align*}
    F(\bar{\vect{x}}(t+1)) - F(\bar{\vect{x}}(t)) \geq
   \nabla F(\bar{\vect{x}}(t)).(\bar{\vect{x}}(t+1)-\bar{\vect{x}}(t))-\frac{1}{2}N^2\frac{1}{T^2}f(\mathcal{S}^\star),
\end{align*}
which, given~\eqref{eq::Pro51_b}, leads to
\begin{align}\label{eq::mainthm2}
  F(\bar{\vect{x}}(t+1)) - F(\bar{\vect{x}}(t)) \geq
   \frac{1}{T} \SUM{i \in \mathcal{A}}{} \widetilde{\vect{v}}_i(t).\nabla F(\bar{\vect{x}}(t)) -\frac{1}{2}N^2\frac{1}{T^2}f(\mathcal{S}^\star).
\end{align}
Next, we note that by definition, $\bar{\vect{x}}(t)\geq \vect{x}_i(t)$ for any  $\forall i \in \mathcal{A}$. Therefore, given~\eqref{eq::Pro51_a}, by invoking 
Lemma~\ref{lm::aux3}, for any  $i \in \mathcal{A}$ we can write
\begin{align}\label{eq::mainthm3}
    &\left|\frac{\partial F}{\partial x_p}(\bar{\vect{x}}(t)) - \frac{\partial F}{\partial x_p}({\vect{x}}_i(t))\right|\leq  N\frac{1}{T}d(\mathcal{G}))f(\mathcal{S}^\star),
\end{align}
for $p\in \{1,\cdots,n\}$. Recall that at each time step $t$, the realization of $\widetilde{\vect{v}}_i(t)$ in~\eqref{eq::updateVect} that Algorithm~\ref{alg:the_alg_decen_2} uses is
\begin{align}\label{eq::t_v_i}
\widetilde{\vect{v}}_i(t)=\vect{1}_{p^\star}, \qquad p^\star\in\mathcal{P}_i~~ \text{is given by}~\eqref{eq::p_star} \end{align}
for every $i\in\mathcal{A}$. Thus, $\vect{1}.\widetilde{\vect{v}}_i(t)=1$, $i\in\mathcal{A}$. Consequently, using~\eqref{eq::mainthm3} we can write
\begin{align}\label{eq::inEquality1}
    \SUM{i \in \mathcal{A}}{} \widetilde{\vect{v}}_i(t). \nabla F(\bar{\vect{x}}(t)) \geq
    \SUM{i \in \mathcal{A}}{} \widetilde{\vect{v}}_i(t). \nabla F(\vect{x}_i(t)) - N^2\frac{1}{T}d(\mathcal{G}))f(\mathcal{S}^\star).
\end{align}
Next, we let 
$$\bar{\vect{v}}_i(t) = \underset{\vect{v} \in P_i(\mathcal{M})}{\textup{argmax}} \vect{v}.\nabla F(\bar{\vect{x}}(t))$$ 
and 
$$\hat{\vect{v}}_i(t) = \underset{\vect{v} \in P_i(\mathcal{M})}{\textup{argmax}} \vect{v}.\nabla F(\vect{x}_i(t)).$$ 
Because $f$ is monotone increasing, by virtue of Lemma~\ref{lem::prop_F}, $\frac{\partial F}{\partial x_p}\geq0$, and as such $\bar{\vect{v}}_i(t)=\vect{1}_{\bar{p}}$ and $\hat{\vect{v}}_i(t)=\vect{1}_{\hat{p}}$, where $\bar{p}=\underset{p\in\mathcal{P}_i}{\arg\max} \frac{\partial F (\bar{\vect{x}}(t))}{\partial x_p}$ and $\hat{p}=\underset{p\in\mathcal{P}_i}{\arg\max} \frac{\partial F (\vect{x}_i(t))}{\partial x_p}$.
Therefore, using $$\hat{\vect{v}}_i(t). \nabla F(\vect{x}_i(t)) \geq \bar{\vect{v}}_i(t). \nabla F(\vect{x}_i(t))$$ 
and $$\hat{\vect{v}}_i(t). \nabla F(\vect{x}_i(t)) \geq \tilde{\vect{v}}_i(t). \nabla F(\vect{x}_i(t)),$$
$i\in\mathcal{A}$, and ~\eqref{eq::mainthm3} we can also write 
\begin{subequations}\label{eq::determin_inEquality}
\begin{align}
    &\SUM{i \in \mathcal{A}}{} \hat{\vect{v}}_i(t). \nabla F(\vect{x}_i(t)) \geq \SUM{i \in \mathcal{A}}{} \bar{\vect{v}}_i(t). \nabla F(\vect{x}_i(t))\geq 
    \SUM{i \in \mathcal{A}}{} \bar{\vect{v}}_i(t). \nabla F(\bar{\vect{x}}(t)) - N^2\frac{1}{T}d(\mathcal{G})f(\mathcal{S}^\star),\label{eq::inEquality4}\\
    &\SUM{i \in \mathcal{A}}{} \hat{\vect{v}}_i(t). \nabla F(\vect{x}_i(t)) \geq \SUM{i \in \mathcal{A}}{} \tilde{\vect{v}}_i(t). \nabla F(\vect{x}_i(t)).\label{eq::determin_inEquality_b}
\end{align}
\end{subequations}

On the other hand, by virtue of Lemma~\ref{lem::gradient_estimate}, $\frac{\widetilde{\partial F}}{\partial x_p}(\vect{x}_j(t))$, $p \in \mathcal{P}_j$ that each agent $j\in\mathcal{A}$ uses to solve optimization problem~\eqref{eq::p_star} (equivalently \eqref{eq::updateVect}) satisfies \begin{align}\label{eq::F_tilde_F}
    \left|\frac{\widetilde{\partial F}}{\partial x_p}(\vect{x}_j(t)) - \frac{{\partial F}}{\partial x_p}(\vect{x}_j(t))\right| \leq \frac{1}{2T}f(\mathcal{S}^\star)
\end{align} with the probability of  $1-2\textup{e}^{-\frac{1}{8T^2}K_j}$. Using~\eqref{eq::determin_inEquality_b} and~\eqref{eq::F_tilde_F}, and also that the samples are drawn independently 
\begin{subequations}
\begin{align}\label{eq::inEquality2}
    &\SUM{i \in \mathcal{A}}{} \widetilde{\vect{v}}_i(t). \nabla F(\vect{x}_i(t)) \geq
    \SUM{i \in \mathcal{A}}{} \widetilde{\vect{v}}_i(t). \widetilde{\nabla F}(\vect{x}_i(t)) - N\frac{1}{2T}f(\mathcal{S}^\star),\\
    &\SUM{i \in \mathcal{A}}{} \widetilde{\vect{v}}_i(t). \widetilde{\nabla F}(\vect{x}_i(t)) \geq
    \SUM{i \in \mathcal{A}}{} \hat{\vect{v}}_i(t). \widetilde{\nabla F}(\vect{x}_i(t))\geq 
    \qquad\SUM{i \in \mathcal{A}}{} \hat{\vect{v}}_i(t). \nabla F(\vect{x}_i(t)) - N\frac{1}{2T}f(\mathcal{S}^\star),\label{eq::inEquality3}
\end{align}    
\end{subequations}
with the probability of $\prod_{i\in\mathcal{A}} (1-2\textup{e}^{-\frac{1}{8T^2}K_j})^{|\mathcal{P}_i|}$.

From~\eqref{eq::inEquality1},~\eqref{eq::inEquality4},\eqref{eq::inEquality2}, and \eqref{eq::inEquality3} now we can write 
\begin{align}\label{eq::inEquality5}
    \SUM{i \in \mathcal{A}}{} \widetilde{\vect{v}}_i(t). \nabla F(\bar{\vect{x}}(t)) \geq \SUM{i \in \mathcal{A}}{} \bar{\vect{v}}_i(t). \nabla F(\bar{\vect{x}}(t)) - (2 Nd(\mathcal{G}))+1 )N\frac{1}{T}f(\mathcal{S}^\star),
\end{align}
with the probability of $1-2\,\sum_{i \in \mathcal{A}}^{}\textup{e}^{-\frac{1}{8T^2}K_i}$.

Next, let $\vect{v}_i^\star$ be the projection of $\vect{1}_{\mathcal{S}^\star}$ into $P_i(\mathcal{M})$. Knowing that $P_i(\mathcal{M})$s are disjoint sub-spaces of $P(\mathcal{M})$ covering the whole space then we can write 
\begin{align}\label{eq::subspace}
    \vect{1}_{\mathcal{S}^\star} = \SUM{i \in \mathcal{A}}{} \vect{v}_i^\star.
\end{align}
Then, using \eqref{eq::inEquality5}, \eqref{eq::subspace}, and invoking Lemma \ref{lm::optimalInEq} and the fact that $\bar{\vect{v}}_i(t). \nabla F(\bar{\vect{x}}(t)) \geq \vect{v}^\star_i(t). \nabla F(\bar{\vect{x}}(t))$ we obtain
\begin{align}\label{eq::inEquality6}
    &\SUM{i \in \mathcal{A}}{} \widetilde{\vect{v}}_i(t). \nabla F(\bar{\vect{x}}(t)) \geq \SUM{i \in \mathcal{A}}{} \vect{v}^\star_i(t). \nabla F(\bar{\vect{x}}(t)) - (2 Nd(\mathcal{G}))+1 )N\frac{1}{T}f(\mathcal{S}^\star) =\nonumber \\    
    &\vect{1}_{\mathcal{S}^\star}.\nabla F(\bar{\vect{x}}(t)) - (2 Nd(\mathcal{G}))+1 )N\frac{1}{T}f(\mathcal{S}^\star) \geq f(\mathcal{S}^\star) - F(\bar{\vect{x}}(t))- (2 Nd(\mathcal{G}))+1)\frac{N}{T}f(\mathcal{S}^\star),
\end{align}
with the probability of $\prod_{i\in\mathcal{A}} (1-2\textup{e}^{-\frac{1}{8T^2}K_j})^{|\mathcal{P}_i|}$. Hence, using~\eqref{eq::mainthm2} and \eqref{eq::inEquality6}, we conclude that
\begin{align}\label{eq::mainthm4}
   F(\bar{\vect{x}}(t+1)) - F(\bar{\vect{x}}(t)) \geq  \frac{1}{T}(f(\mathcal{S}^\star) - F(\bar{\vect{x}}(t))-
   (2 Nd(\mathcal{G}))+\frac{1}{2}N+1)\frac{N}{T^2}f(\mathcal{S}^\star),
\end{align} 
with the probability of $\prod_{i\in\mathcal{A}} (1-2\textup{e}^{-\frac{1}{8T^2}K_j})^{|\mathcal{P}_i|}$.

Next, let $g(t)=f(\mathcal{S}^\star)-F(\bar{\vect{x}}(t))$ and $\beta =(2 Nd(\mathcal{G}))+\frac{1}{2}N+1)\frac{N}{T^2}f(\mathcal{S}^\star)$, to rewrite~\eqref{eq::mainthm4} as 
\begin{align}\label{eq::mainthm5}
   &(f(\mathcal{S}^\star)-F(\bar{\vect{x}}(t))) - (f(\mathcal{S}^\star)-F(\bar{\vect{x}}(t+1)))=\nonumber \\
   &g(t)-g(t+1) \geq
   \frac{1}{T}(f(\mathcal{S}^\star)-F(\bar{\vect{x}}(t))) - \beta  = \frac{1}{T}g(t)-\beta.
\end{align}\label{eq::mainthm6}
Then from inequality \eqref{eq::mainthm5} we get 
\begin{align}\label{eq::mainthm6}
   g(t+1) \leq (1-\frac{1}{T})g(t)+\beta
\end{align}
with the probability of $\prod_{i\in\mathcal{A}} (1-2\textup{e}^{-\frac{1}{8T^2}K_j})^{|\mathcal{P}_i|}$. Solving for inequality \eqref{eq::mainthm6} at time $T$ yields
\begin{align}\label{eq::mainthm7}
    g(T) \leq (1-\frac{1}{T})^T g(0) +\beta \SUM{k=0}{T-1} (1-\frac{1}{T})^k=(1-\frac{1}{T})^T g(0) + T\beta(1-(1-\frac{1}{T})^T)
\end{align}
with the probability of 
$\left(\prod_{i\in\mathcal{A}} (1-2\textup{e}^{-\frac{1}{8T^2}K_j})^{|\mathcal{P}_i|}\right)^T$.
Substituting back $g(T)=f(\mathcal{S}^\star)-F(\bar{\vect{x}}(T))$ and $g(0)=f(\mathcal{S}^\star)-F(\vect{x}(0))=f(\mathcal{S}^\star)$, in~\eqref{eq::mainthm7} we then obtain 
\begin{align}\label{eq::mainthm8}
    (1-(1-\frac{1}{T})^T)(f(\mathcal{S}^\star)-T\beta)=(1-(1&-\frac{1}{T})^T)(1-(2 Nd(\mathcal{G}))+\frac{1}{2}N+1)\frac{N}{T})f(\mathcal{S}^\star)\nonumber \\
    &\leq F(\bar{\vect{x}}(T)),
\end{align}
with the probability of $\left(\prod_{i\in\mathcal{A}} (1-2\textup{e}^{-\frac{1}{8T^2}K_j})^{|\mathcal{P}_i|}\right)^T$.
By applying $\frac{1}{\textup{e}}\geq(1-(1-\frac{1}{T})^T)$, we get 
\begin{align}
(1\!-\!\frac{1}{\textup{e}})(1\!-\!(2 Nd(\mathcal{G})\!+\!\frac{1}{2}N+1)\frac{N}{T})f(\mathcal{S}^\star) \!\leq F(\bar{\vect{x}}(T)),
\end{align}
with the probability of $\left(\prod_{i\in\mathcal{A}} (1-2\textup{e}^{-\frac{1}{8T^2}K_j})^{|\mathcal{P}_i|}\right)^T$.

Given~\eqref{eq::t_v_i}, from the propagation and update rules \eqref{eq::updateRule1} and \eqref{eq::updateRule2} and Lemma~\ref{lem::maxVect} we can conclude that $\vect{1}.\vect{x}_{ii}(T)=1$. Furthermore by defining $\mathcal{R}_{\vect{x}_{ii}(T)}$ to be a random set where each member if sampled according to $\vect{x}_{ii}(T)$ and from $\mathcal{P}_i$. Since $\vect{1}.\vect{x}_{ii}(T)=1$, we can also define $\mathcal{T}_{\vect{x}_{ii}(T)}$ to be a random set where only one policy is sampled from $\mathcal{P}_i$ according to $\vect{x}_{ii}(T)$, then using Lemma \ref{lm::expectedInEq}, we can write
\begin{align*}
    &F(\bar{\vect{x}}(T)) = \mathbb{E}[f(\mathcal{R}_{\vect{x}(T)})]=\mathbb{E}[f(\mathcal{R}_{\vect{x}_{11}(T)} \cup \cdots \cup \mathcal{R}_{\vect{x}_{NN}(T)})]\\
    &\leq \mathbb{E}[f(\mathcal{T}_{\vect{x}_{11}(T)} \cup \mathcal{R}_{\vect{x}_{22}(T)} \cup \cdots \cup \mathcal{R}_{\vect{x}_{NN}(T)})]\\
    &\leq \cdots \\
    &\leq \mathbb{E}[f(\mathcal{T}_{\vect{x}_{11}(T)} \cup \cdots \cup \mathcal{T}_{\vect{x}_{NN}(T)})]
\end{align*}
which concludes the proof.
\end{proof}

\section*{Appendix B: Auxiliary Results}
In what follows we let  $\mathcal{S}^\star$ to represent the optimizer of problem~\eqref{eq::mainProblem}.
\begin{lem}\label{lm::optimalInEq}
{\it
Consider the set value optimization problem~\eqref{eq::mainProblem}. Suppose $f:\mathcal{P} \to \realnonnegative$ is an increasing and submodular set function and consider its multilinear extension $F:\realnonnegative^n \to \realnonnegative$. Then $\forall \vect{x} \in P(\mathcal{M})$
\begin{align*}
    \vect{1}_{\mathcal{S}^\star}.\nabla F(\vect{x}) \geq f(\mathcal{S}^\star) - F(\vect{x}).
\end{align*}
}\end{lem}
\begin{proof}
see~\cite{JV:08} for the proof.
\end{proof}

\begin{lem}\label{lm::secondDer}
\it
Consider the set value optimization problem~\eqref{eq::mainProblem}. Suppose $f:\mathcal{P} \to \realnonnegative$ is an increasing and submodular set function and consider its multilinear extension $F:\realnonnegative^n \to \realnonnegative$. Then  
\begin{align*}
\left|\frac{\partial^2F}{\partial x_p \partial x_q}\right| \leq f(\mathcal{S}^\star),\qquad p,q\in\mathcal{P}.
\end{align*}
\end{lem}
\begin{proof}
Since $p \not \in \mathcal{R}_{\vect{x}}  \cup \{q\} \setminus \{p\}$, therefor by submodularity of $f$ we can write
\begin{align}\label{eq::partial1}
    &0\leq \Delta_f(p|\mathcal{R}_{\vect{x}}  \cup \{q\} \setminus \{p\})\leq f(\{p\}),\nonumber \\
    &0\leq \Delta_f(p|\mathcal{R}_{\vect{x}} \!\! \setminus \!\! \{p,q\})\leq f(\{p\}).
\end{align}
Knowing that
\begin{align*}
    \Delta_f(p|\mathcal{R}_{\vect{x}} \! \cup \! \{q\} \! \setminus \! \{p\})\!=\!f(\mathcal{R}_{\vect{x}} \cup \{p,q\})\! -\!f(\mathcal{R}_{\vect{x}}  \cup \{q\} \!\! \setminus \!\! \{p\}),
\end{align*}
and
\begin{align*}
\Delta_f(p|\mathcal{R}_{\vect{x}} \!\! \setminus \!\! \{p,q\})=f(\mathcal{R}_{\vect{x}} \cup \{p\}) \!\! \setminus \!\! \{q\}) -f(\mathcal{R}_{\vect{x}} \!\! \setminus \!\! \{p,j\}),
\end{align*}
then definition \eqref{eq::secondDer} can be written as 
\begin{align}\label{eq::partial3}
\frac{\partial F}{\partial x_p \partial x_q} = \mathbb{E}[\Delta_f(p|\mathcal{R}_{\vect{x}}  \cup \{q\} \!\! \setminus \!\! \{p\}) \!\!- \!\! \Delta_f(p|\mathcal{R}_{\vect{x}} \!\! \setminus \!\! \{p,q\})].
\end{align}
Putting \eqref{eq::partial1} and \eqref{eq::partial3} together results in
\begin{align*}
    \left|\frac{\partial^2F}{\partial x_p \partial x_q}\right| \leq f(\{p\}).
\end{align*}
where knowing that $f(\{p\})\leq f(\mathcal{S}^\star)$ concludes the proof.
\end{proof}

\begin{lem}\label{lm::aux3}
Consider a twice differentiable function $F(\vect{x}):\mathbb{R}^N \to \mathbb{R}$ which satisfies $\left|\frac{\partial^2F}{\partial x_p \partial x_q}\right| \leq \alpha$ for any $p,q\in\{1,\cdots,N\}$.  Then for any $\vect{x}_1,\vect{x}_2 \in \mathbb{R}^N$ satisfying $\vect{x}_2 \geq \vect{x}_1$ and $\vect{1}.({\vect{x}}_2 - {\vect{x}}_1) \leq \beta$ the following holds,
\begin{subequations}\label{eq::mainthm1}
\begin{align}
   & \left|\frac{\partial F}{\partial x_p}({\vect{x}}_1+\epsilon({\vect{x}}_2-{\vect{x}}_1)) -  \frac{\partial F}{\partial x_p}({\vect{x}}_1)\right| \leq \epsilon \alpha \beta,\\
   & F({\vect{x}}_2) - F({\vect{x}}_1) \geq 
   \nabla F({\vect{x}}_1).({\vect{x}}_2-{\vect{x}}_1)-\frac{1}{2}\alpha \beta^2,
\end{align}     
\end{subequations}
for $\epsilon \in [0 \,\, 1]$.
\end{lem}

\begin{proof}
Let $\vect{h}_p=[\frac{\partial^2F}{\partial x_p \partial x_1},\cdots,\frac{\partial^2F}{\partial x_p \partial x_N}]^\top$. Then, we can write 
\begin{align}\label{eq::aux3_1}
    &\left|\frac{\partial F}{\partial x_p}({\vect{x}}_1+\epsilon({\vect{x}}_2-{\vect{x}}_1)) -  \frac{\partial F}{\partial x_p}({\vect{x}}_1)\right| =\left|\int_{0}^{\epsilon}
    \vect{h}_l({\vect{x}}_1+\tau({\vect{x}}_2-{\vect{x}}_1)).({\vect{x}}_2-{\vect{x}}_1) \textup{d}\tau\right| \nonumber \\
    &\leq \int_{0}^{\epsilon} \alpha \vect{1}.({\vect{x}}_2-{\vect{x}}_1)\textup{d}\tau = \epsilon \alpha \beta,
\end{align}

Furthermore,
\begin{align*}
   & F({\vect{x}}_2) - F({\vect{x}}_1) = \int_{0}^{1}
   \nabla F({\vect{x}}_1+\epsilon({\vect{x}}_2-{\vect{x}}_1)).(\bar{\vect{x}}(t+1)-\bar{\vect{x}}(t))\textup{d}\epsilon  \nonumber \\ 
   &\geq \int_{0}^{1}
   (\nabla F({\vect{x}}_1)-\epsilon \alpha \beta).({\vect{x}}_2-{\vect{x}}_1)\textup{d}\epsilon = 
   \nabla F({\vect{x}}_1).({\vect{x}}_2\!\!-\!\!{\vect{x}}_1)-\frac{1}{2}\alpha \beta^2,
\end{align*}
with the third line follow from equation~\eqref{eq::aux3_1}, which concludes the proof. 
\end{proof}
\begin{lem}\label{lem::gradient_estimate}{\it
Consider the set value optimization problem~\eqref{eq::mainProblem}. Suppose $f:\mathcal{P} \to \realnonnegative$ is an increasing and submodular set function and consider its multilinear extension $F:\realnonnegative^n \to \realnonnegative$. Let $\widetilde{\nabla F}(\vect{x})$ be the  estimate of $\nabla F(\vect{x})$ that is calculated by taking $K\in\real_{>0}$ samples of set $\mathcal{R}_{\vect{x}}$ according to membership probability vector $\vect{x}$.  Then 
\begin{align}\label{eq::estimationACC}
    \left|\frac{\widetilde{\partial F}}{\partial x_p}(\vect{x}) - \frac{{\partial F}}{\partial x_p}(\vect{x})\right| \geq \frac{1}{2T}f(\mathcal{S}^\star), \quad p\in \{1,\cdots,n\},
\end{align}
with the probability of $2\textup{e}^{-\frac{1}{8T^2}K}$, for any $T\in\real_{>0}$.
}
\end{lem}
\begin{proof}
Define the random variable 
\begin{align*}
  \!  X\!\! =\!\! \left(\!\!(f(\mathcal{R}_{\vect{x}}\!\cup \{p\})\!-\!f(\mathcal{R}_{\vect{x}}\!\!\setminus\! \{p\}))\!-\!\frac{\partial F}{\partial x_{p}}(\vect{x})\!\!\right)\!\!\Big/\!\!f(\mathcal{S}^\star),
\end{align*}
and assume that agent $j\in\mathcal{A}$ takes $K_j$ samples from  $\mathcal{R}_{\vect{x}}$ to construct $\{X_k\}_{k=1}^{K}$ realization of $X$. Since $f$ is a submodular function, then we have $(f(\mathcal{R}_{\vect{x}}\!\cup \{p\})\!-\!f(\mathcal{R}_{\vect{x}}\!\!\setminus\! \{p\}))\leq f(\mathcal{S}^\star)$. Consequently using equation~\eqref{eq::firstDer}, we conclude that $0\leq X \leq 1$. Hence, using Theorem~\ref{thm::sampling}, we have
\begin{align*}
    \left|\SUM{k=1}{K}X_k\right|\geq\frac{1}{2T}K.
\end{align*}
with the probability of $2\textup{e}^{-\frac{1}{8T^2}K}$.  Hence, the estimation accuracy of $\nabla F(\vect{x})$, is given by
\begin{align*}\label{eq::estimationACC}
    \left|\frac{\widetilde{\partial F}}{\partial x_p}(\vect{x}) - \frac{{\partial F}}{\partial x_p}(\vect{x})\right| \geq \frac{1}{2T}f(\mathcal{S}^\star), \quad p\in \{1,\cdots,n\}.
\end{align*}
with the probability of $2\textup{e}^{-\frac{1}{8T^2}K}$.
\end{proof}


\begin{lem}\label{lm::expectedInEq}
\it
Suppose $f:\mathcal{P} \to \realnonnegative$ is an increasing and submodular set function and consider $\vect{x}$ to be a membership probability vector of set $\mathcal{Q} \subset \mathcal{P}$ with $\vect{1}.\vect{x} = 1$. We define $\mathcal{R}_{\vect{x}}$ to be the set resulted by sampling independently each member of $\mathcal{Q}$ according to probability vector $\vect{x}$ and $\mathcal{T}_{\vect{x}}=\{t\}, \,\, t \in \mathcal{Q}$ to be a single member set which is chosen according to $\vect{x}$. the following holds for any random set $\mathcal{S} \in \mathcal{P} \setminus \mathcal{Q}$. 
\begin{align*}
    \mathbb{E}[f(\mathcal{R}_{\vect{x}_{ii}} \cup S)] \leq \mathbb{E}[f(\mathcal{T}_{\vect{x}_{ii}} \cup S)].
\end{align*}
\end{lem}
\begin{proof}
Defining $\mathcal{R}_{\vect{x}}=\{r_1,\cdots,r_o\}$, then we can write 
\begin{align*}
    &\mathbb{E}[f(\mathcal{R}_{\vect{x}} \cup \mathcal{S})]=\mathbb{E}[f(\mathcal{S})+\SUM{l=1}{o}\Delta_f(r_l|\mathcal{S} \cup \{r_1,\cdots, r_{l-1}\})] \\
    &\leq \mathbb{E}[f(\mathcal{S})+\SUM{r_l \in  \mathcal{R}_{\vect{x}} }{}\Delta_f(r_l|\mathcal{S})]=\mathbb{E}_{\mathcal{S}}[\mathbb{E}_{\mathcal{R}_{\vect{x}}|\mathcal{S}}[f(\mathcal{S})+\SUM{r_l \in  \mathcal{R}_{\vect{x}} }{}\Delta_f(r_l|\mathcal{S})]]\\
    &=\mathbb{E}_{\mathcal{S}}[f(\mathcal{S})+\mathbb{E}_{\mathcal{R}_{\vect{x}}|\mathcal{S}}[\SUM{r_l \in  \mathcal{R}_{\vect{x}} }{}\Delta_f(r_l|\mathcal{S})]]=\mathbb{E}_{\mathcal{S}}[f(\mathcal{S})+\SUM{r_l \in  \mathcal{Q} }{}x_{l}\Delta_f(r_l|\mathcal{S})]\\
    &=\mathbb{E}_{\mathcal{S}}[f(\mathcal{S})+\mathbb{E}_{\mathcal{T}_{\vect{x}}|\mathcal{S}}[\Delta_f(t|\mathcal{S})]]=\mathbb{E}_{\mathcal{S}}[\mathbb{E}_{\mathcal{T}_{\vect{x}}|\mathcal{S}}[f(\mathcal{S})+\Delta_f(t|\mathcal{S})]]=\mathbb{E}[f(\mathcal{T}_{\vect{x}} \cup S)]
\end{align*}
which concludes the proof. 
\end{proof}

\end{document}